\documentclass[11pt,reqno,fleqn]{amsart}

\usepackage{amsfonts,amsmath,amsthm,amssymb}
\usepackage{pdfsync}

\newtheorem{thm}{Theorem}

\newtheorem{lemma}{Lemma}
\newtheorem{prop}{Proposition}
\newtheorem{cor}{Corollary}

\newtheorem{claim}{Claim}
\newtheorem*{thmA}{Theorem A}
\newtheorem*{thmB}{Theorem B}
\newtheorem*{thmC}{Theorem C}

\newcommand{\R}{\ensuremath{\mathbb{R}}}

\newcommand{\N}{\ensuremath{\mathbb{N}}}
\newcommand{\F}{\ensuremath{\mathbb{F}}}
\newcommand{\EE}{\ensuremath{\mathbb{E}}}
\newcommand{\ba}{\ensuremath{\mathbf{a}}}
\newcommand{\bF}{\ensuremath{\bar{\mathbb{F}}}}

\newcommand{\1}{\ensuremath{\mathbf{1}}}

\newcommand{\subs}{\ensuremath{\subseteq}}

\newcommand{\eps}{\ensuremath{\varepsilon}}
\newcommand{\la}{\ensuremath{\lambda}}

\newcommand{\al}{\ensuremath{\alpha}}
\newcommand{\be}{\ensuremath{\beta}}
\newcommand{\ga}{\ensuremath{\gamma}}

\newcommand{\de}{\ensuremath{\delta}}

\newcommand{\eq}{\begin{equation}
\newcommand{\ee}{\end{equation}}}

\numberwithin{equation}{section}

\begin{document}
\title{On restricted arithmetic progressions\\ over finite fields}
\author{Brian Cook\quad\quad\quad\'Akos Magyar}
\thanks{The second author was supported by NSERC Grant 22R44824.}

\address{Department of Mathematics, The University of British Columbia, Vancouver, BC, V6T1Z2, Canada}
\email{bcook@math.ubc.ca}
\address{Department of Mathematics, University of British Columbia, Vancouver, B.C. V6T 1Z2, Canada}
\email{magyar@math.ubc.ca}

\begin{abstract} Let $A$ be a subset of $\F_p^n$, the $n$-dimensional linear space over the prime field $\F_p$ of size at least $\de N$ $(N=p^n)$, and let $S_v=P^{-1}(v)$ be the level set of a homogeneous polynomial map $P:\F_p^n\to\F_p^R$ of degree $d$, and $v\in\F_p^R$. We show, that under appropriate conditions, the set $A$ contains at least $c\, N|S|$ arithmetic progressions of length $l\leq d$ with common difference in $S_v$, where c is a positive constant depending on $\de$, $l$ and $P$. We also show that the conditions are generic for a class of sparse algebraic sets of density $\approx N^{-\eps}$.
\end{abstract}

\maketitle

\section{Introduction.}
\subsection{Background.}

A famous result of Szemer\'edi \cite{SZ} states that a set $A$ of positive upper density of the integers contains arbitrarily long arithmetic progressions $x,x+d,\ldots,x+ld$. There has been many generalizations and extensions, a natural question one may ask is whether one may add restrictions on the common difference $d$. A typical example of this type is the well-known theorem of S\'ark\"ozy \cite{S}, saying that $A$ contains two elements whose difference is a square. More recently Green has shown \cite{G} that $A$ contains a 3-term arithmetic progression, whose common difference is a sum of two squares. Far reaching results of this type for longer progressions have been obtained recently by Green and Tao \cite{GT} and by Wooley and Ziegler \cite{WZ} where the gap is of the form $p-1$ and $P(p-1)$, $p$ being a prime and $P$ an integral polynomial such that $P(0)=0$.

The aim of this note is to provide a simple extension of this type in the finite field settings, where $A\subs\F_p^n$ is a set of density $\de>0$ and one is counting arithmetic progressions in $A$ with gaps in algebraic sets $S$ given as level sets of a family of homogeneous polynomials.

\subsection{Main results.} Let $\F_p^n$ be the $n$-dimensional linear space above the prime field $\F_p$, and for a fixed $\de>0$ let $A\subs\F_p^n$ be a set of at least $|A|\geq\de p^n$ elements. The finite field version of Szemer\'edi's theorem states that such sets will contain genuine arithmetic progressions of length $l\leq p$ as long as $n$ is large enough. Note that the condition $l\leq p$ is natural as it ensures that progressions in the form $\{x,x+y,\ldots,x+(l-1)y\}$ consist of distinct points for $y\neq 0$. We will need the following quantitative version

\begin{thmA} Let $\de>0$ and $l\in\N$. For a function $f:\F_p^n\to [0,1]$ satisfying $\EE(f(x):\,x\in\F_p^n)\geq\de$ one has that
\eq \EE(f(x)f(x+y)\ldots f(x+(l-1)y):\ x,y\in\F_p^n)\geq c(\de,l,p),\ee
where $c(\de,l,p)$ is a positive constant depending only on $\de$, $l$ and $p$.
\end{thmA}

Let $S_v=P^{-1}(v)$ be an algebraic set defined as the level set of a family of homogeneous polynomials $P=(P_1,\ldots,P_R):\F_p^n\to\F_p^R$ of degree $d$, $v\in\F_p^R$ being a given vector. We will be interested in counting $l$-term arithmetic progressions in $A$ with common difference $y$ in $S_v$. It is clear that in order to make this problem well-defined one needs to make a few assumptions. First $S_v$ needs to be nonempty, and in order to have a more precise formula for the size of $S_v$, a natural assumption is that the associated set of singular points
\eq S_P^*:=\{x\in\F_p^n:\ rank\,(Jac_P (x))<R,\}\ee
is small. Here $Jac_P (x)$ is the $R\times n$ matrix with entries $\partial_{x_j}P_i(x)$. This will be done by requiring that
\eq K:= codim (S_P^*)\ee is sufficiently large. Note that the dimension of an algebraic set is defined above the algebraic closure of the prime field $\F_p$.
To avoid degeneracies like the identical vanishing of certain derivatives we'll also assume that $d<p$. Let us introduce the parameters $0<\al,\be,\ga<1$ by
\eq \ga n=R,\ \be n=K,\ p^\al=d\ee
Our main result is the following.

\begin{thm}\label{main} Let $\de>0$, $\eps>0$ be given, and let $A\subs\F_p^n$ be a set of size $|A|\geq\de p^n$. Let the polynomial map $P$ and the parameters $0<\al,\be,\ga<1$ be defined as above. Then for $l\in\N$, $l\leq d$ one has uniformly in $v\in\F_p^R$ that
\eq \EE(\1_A(x)\1_A(x+y)\ldots \1_A(x+(l-1)y):\ x\in\F_p^n,\,y\in S_v)\geq c(\eps,\de,l,p),\ee
provided that
\eq \be-\al-(2^d+1)\ga\geq\eps,\ee
with a constant $c(\eps,\de,l,p)>0$ depending only on $\eps,\de,l$ and $p$.
Here $\1_A$ stands for the indicator function of the set $A$, and $S_v=P^{-1}(v)$.
\end{thm}

\noindent{\underline{Remarks:}}

\begin{itemize}

\item If $n$ is large enough, it follows that $A$ contains $\approx |\F_p^n||S_v|$ non-trivial progressions of length $l$ with common difference $y\in S_v$. In particular for every $v\in\F_p^R$ there is a progression with common difference $y$ such that $P(y)=v$. As a byproduct of the proof we also obtain that $|S_v|\approx p^{n-R}$, uniformly in $v$, thus is a "sparse" set of density of $\approx p^{-\ga n}$.

\medskip

\item The condition $l\leq d$ seems necessary, as it can be seen from the following example in \cite{BO} adapted to the finite field settings. Let $d=2$, $R=1$ and $P(x)=x_1^2+\ldots +x_n^2$. Fix $p>2$ (say $p=5$) and let $A=\{x\in\F_p^n:\,P(x)=0\}$, then $A$ has density $\approx p^{-1}$. By the parallelogram identity: $P(x)-2P(x+y)+P(x+2y)=2P(y)$, if $A$ contains a 3-term arithmetic progression $\{x,x+y,x+2y\}$ then necessarily $P(y)=0$. One can construct similar examples for the polynomials $Q(x)=\sum_j x_j^d$ of degree $d$ for all $d\geq 2$. These examples also show the necessity of a condition on the singular set. Indeed (1.5) does not hold for the level sets of $P(x)=(x_1^2+\ldots,+x_n^2)^{d/2}$ ($d>2$ even) for $l>2$, while it dose hold for the level sets of $Q(x)=x_1^d+\ldots ,x_n^d$ for $l=d$. The difference is that $S_P^*$ is $n-1$-dimensional while $S_Q^*=\{0\}$.

\medskip

\item For the special case, when $v=0$ the set $A$ contains non-trivial arithmetic progressions with gap $y\in S=P^{-1}(0)$ of length $l>d$, under the more restrictive conditions that $R\leq c(\de,l,p,d)\,n$. This is based on the fact that zero set of homogeneous polynomial maps contain a large linear subspace, which follows from a theorem of Chevalley and Warning \cite{FF}, see Thm. 6.11. This will be discussed in Section 4.

\end{itemize}

\medskip

We will also study polynomial maps $P$ for which the conditions of Theorem 1 hold. Consider first diagonal forms, when $P_i(x)=\sum_{j=1}^n a_{ij}x_j^d$, $A=\{a_{ij}\}_{1\leq i\leq R,1\leq j\leq n}$ being an $R\times n$ matrix. We say that the matrix $A$ is {\em non-degenerate} if $rank\,A'=R$ for every $R\times n/2$ submatrix of $A$.

\begin{claim} If $A$ is non-degenerate then $dim\,S_P^*\leq n/2$.
\end{claim}

\begin{proof} For $1\leq j\leq n$ let $\ba_j$ be the $j$-th column of the matrix $A$. Then the $j$th column of the Jacobian $Jac_P(x)$ is $dx_j^{d-1}\ba_j$ at $x=(x_1,\ldots,x_n)$. If $x$ has at least $n/2$ nonzero coordinates, say $x_{j_1},\ldots,x_{j_m}$ then the corresponding columns of $Jac_P(x)$ span $\F_p^R$, hence $rank\,(Jac_P(x))=R$. Thus $S_P^*$ is contained in the union of the $n/2$-dimensional coordinate hyperplanes.
\end{proof}

It is easy to see that most $R\times n$ matrices are non-degenerate. Let $A$ be a random matrix obtained by choosing each of its column vector $\ba_i$ independently with probability $p^{-R}$.

\begin{claim} Let $\ga_0:=\frac{1}{2}-\frac{\log\,2}{\log\,5}$. If $p\geq 5$, $R=\ga\,n$ with $\ga<\ga_0$, then the probability that a random matrix $A$ is non-degenerate is at least $1-p^{-(\ga_0-\ga)n}$.
\end{claim}

\begin{proof} For a given subspace $M\leq \F_p^n$ of codimension 1, the probability that the rows $\ba_{i_1},\ba_{i_2},\ldots,\ba_{i_m}$ of the random matrix $A$ are all contained in $M$ is $p^{-m}$. Thus the probability that $M$ contains at least $n/2$ columns of $A$ is less than $2^n p^{-n/2}$. Since there are $p^R$ distinct $R-1$ dimensional subspaces, the probability that none of them will contain at least $n/2$ columns of $A$ is at least
\[1-2^n p^{R-\frac{n}{2}} \geq 1-p^{-(\ga_0-\ga)n}.\]
In that case $A$ is non-degenerate.
\end{proof}

This implies that condition (1.6) holds for such maps as long as $\al+(2^d+1)\ga\leq 1/2-\eps$.
In fact we will show that it also holds for generic (non-diagonal) polynomial maps as long as $\al$ and $\ga$ is chosen sufficiently small. More precisely, let $\mathcal{P}(n,d)$ be the space of homogeneous polynomials $P:\bF_p^n\to\bF_p$ of degree $d$, which is an $N={n+d-1 \choose d}$ dimensional linear space over $\bF_p$, the algebraically closed field of characteristic $p$. Then $\bF(n,d)^R$ is the space of $R$-tuples of such polynomials. One has

\begin{prop}\label{generic} Let $1\leq R\leq L\leq n$ be given. Then the locus of polynomial maps $P\in \bF(n,d)^R$ such that $dim\,S_P^*\geq L$ is contained in an algebraic set of codimension at least $L-2R+2$.
\end{prop}

This mean that for generic polynomial maps $P$ the dimension of the singular variety satisfies the bound: $dim\,S_P^*\leq 2R-2$. If $R=\ga\,n$ then for generic maps one can take $\be=1-2\ga$, thus (1.6) holds if: $\al+(2^d+3)\ga\leq 1-\eps$.

\subsection{Outline of the Proof.} To begin we shall need a generalized von Neumann inequality for restricted progressions. To state it, let us briefly recall the Gowers uniformity norms. The multiplicative derivative of a
function $f$ is given by $\Delta_hf(x)=f(x+h)\overline{f(x)}$, and
higher derivatives as
$\Delta_{h_1,...,h_l}=\Delta_{h_l}(\Delta_{h_1,...,h_{l-1}})$. The
$U^l$-Gowers norm is then given by
\[ ||f||_{U^l}^{2^l}=\mathbb{E}(\Delta_{h_1,...,h_l}f(x):x,h_1,...,h_l\in
\mathbb{F}_p^n).\] This norm represents the average of $f$ over
'cubes' in $\mathbb{F}_p^n$, sets of the form
\[\{x+\omega_1h_1+...+\omega_lh_l:\omega_1,...,\omega_l\in\{0,1\}^l\},\]
and is indeed a norm for integers $l>1$ ($l=1$ provides a
semi-norm). On fact we shall need is the monotonicity formula
\eq ||f||_{U^{l-1}}\leq||f||_{U^{l}}.\ee
For the above definitions and facts on may consult \cite{taovu}.
For a pair of functions $f,g:\F_p^n\to\mathbb{C}$ define the form
\[
\widetilde{\Lambda_l}(f,g)=\mathbb{E}(f(x)f(x+r)...f(x+(l-1)r)g(r):x,r\in\mathbb{F}_p^n).\]

\begin{lemma} (Generalized von Neumann inequality) \label{vN1}

For functions $f,g$ bounded in absolute value by one, one has
\[\widetilde{\Lambda_l}(f,g)\leq ||g||_{U^l}.\]
\end{lemma}

\medskip

To apply this result in combination with Theorem A, one finds
an appropriate balanced function of the level set $S=P^{-1}(v)$, say
$g=\mathbf{1}_S-\rho$ for an appropriate constant $\rho$, and writes
\eq
\widetilde{\Lambda_l}(f,\mathbf{1}_S)=\rho\Lambda_lf+\widetilde{\Lambda_l}(f,g).\ee
The first term applies to Theorem A, while the second can be
bounded by $||g||_{U^l}$ by Lemma \ref{vN1}. Then it remains to show
that $\rho$ can be chosen properly as to give $||g||_{U^l}$ small. Thus the crucial step is to obtain the following bound, which may be of interest on its own.

\begin{prop}\label{norm} Let $v\in\F_p^R$ and let $S=P^{-1}(v)$, where $P=(P_1,\ldots,P_R):\F_p^n\to\F_p^R$ is a homogeneous polynomial map of degree $d$. Then one has
\eq ||\mathbf{1}_S-p^{-R}||_{U^d}\leq (d-1)^{2^{-d}n}\
p^{2^{-d}(R-K)},\ee
where $K=codim (S_P^*)$, $S_P^*$ being the singular variety associated to $P$.
\end{prop}

\begin{proof}[\underline{Proof of Theorem \ref{main}}] Let the parameters $0<\al,\be,\ga <1$ be defined as in (1.4) and assume that condition (1.6) holds. First, note that by (1.9) and (1.6) we have that
\[\left|p^{-n}|S|-p^{-R}\right|=\|\1_S-p^{-R}\|_{U^1}\leq p^{n(\al+\ga-\be)2^{-d}}\leq p^{-\eps n}p^{-R}\]
thus in particular $|S|=p^{n-R}(1+O(p^{-\eps n}))$.

\medskip

\noindent Let $f=\1_A$, $\rho=p^{-R}$, $g=\1_S-p^{-R}$, then by (1.8)
\[\EE\,(\1_A(x)\1_A(x+y)\ldots \1_A(x+(l-1)y):\ x\in\F_p^n,\,y\in S)=\]
\[= p^n |S|^{-1}\,( p^{-R}\Lambda_lf+ \widetilde{\Lambda_l}(f,g)) = (1+O(p^{-\eps n}))\,(\Lambda_lf+p^R \widetilde{\Lambda_l}(f,g))\]

\noindent By Theorem A the first term satisfies
\[\Lambda_lf\geq c(\de,l,p)\]
while by (1.9) and (1.6) the second term is at most
\[p^R \widetilde{\Lambda_l}(f,g))\leq p^R p^{n(\al+\ga-\be)2^{-d}}\leq p^{-\eps n}\]
This implies that the left side of (1.5) is at least $c(\de,l,p)/2$ for $n\geq n(\eps,\de,l,p)$, while it is trivially at least $p^{-n}$ for all $n$'s. This proves Theorem 1.
\end{proof}

It remains to prove Proposition \ref{norm}. The starting point is the identity\[
\mathbf{1}_S(x)=\mathbb{E}(e(\alpha\cdot
(P(x)-v):\alpha\in\mathbb{F}_p^R)=p^{-R}+
p^{-R} \sum_{\alpha\in\mathbb{F}_p^R,\,\alpha\neq0} e(-\al\cdot v)\,e(\alpha\cdot
P(x)).\] The triangle
inequality for the Gowers norms then gives

\[||g||_{U^d}\leq p^{-R} \sum_{\alpha\in\mathbb{F}_p^R,\,\alpha\neq0} (||e(\alpha\cdot
P(x))||_{U^d},\] reducing our
task to bounding $||e(\alpha\cdot P(x))||_{U^d}$ for a nonzero
$\alpha$.

\medskip

We apply the method of Birch \cite{Birch} to achieve a bound in terms of an explicitly given algebraic set $W^*\subs\F_p^{(d-1)n}$, which will be discussed in detail in the next section.

\medskip

\begin{lemma}~\label{norm}
If $\alpha\neq0$ in $\mathbb{F}^n_p$, then we have
\eq
||e(\alpha\cdot P(\cdot)||_{U^d}^{2^d}\leq
\frac{|W^*|}{p^{(d-1)n}}.\ee
\end{lemma}

The set $W^*$ appears in the work of Birch on exponential sums \cite{Birch} and is closely related to the singular set $S_P^*$, see (2.6) below. In particular if one embeds $\mathbb{F}_p^n$ into the diagonal
$\Delta\subseteq\mathbb{F}_p^{(d-1)n}$ via the map $\Phi(h)=(h,\ldots,h)$,
then $\Phi(S_P^*)=W^*\bigcap \Delta$. As $\Delta$ is a linear subspace
of codimension $(d-2)n$, it follows that $dim(W^*)\leq
(d-2)n+dim(S_P^*)$. Thus, $codim(W^*)\geq codim(S_P^*)=K$.
Also $W^*$ can be partitioned into algebraic sets $W^*_\la$, ($\la\in\F_p^R\backslash\{0\}$) such that
$W^*_\la$ is defined by $n$ equations of degree $d-1$. Then basic facts from algebraic geometry give

\begin{lemma}\label{Wbound}
With $W^*$ as above, we have $|W^*|\leq (d-1)^n p^R\,p^{(d-1)n-K}$.
\end{lemma}

Thus Proposition \ref{norm} follows save for the proofs of Lemma 2 and Lemma 3.

\section{Exponential sum estimates.}

Let $P=(P_1,\ldots,P_R):\F_p^n\to\F_p^R$ be a polynomial map, $P_i$ being a homogeneous polynomial of degree $d$, written in the symmetric form
\eq P_i(x)=\sum_{1\leq j_1,\ldots,j_d\leq n} a^i_{j_1\ldots j_d}\,x_{j_1}\ldots x_{j_d},\ \ \ \ \ x=(x_1,\ldots,x_n)\ee
where $a^i_{j_1\ldots j_d}=a^i_{\pi(j_1)\ldots \pi(j_d)}$ for any permutation $\pi:\{1,\ldots,d\}\to \{1,\ldots,d\}$. Note that this is possible as $(d!,p)=1$. For $h\in\F_p^n$ define the differencing operator
\eq D_h P_i(x)=P_i(x+h)-P_i(x),\ee
and note that $deg\,(D_h P)=deg\,(P)-1$. After applying the differencing operators $d-1$ times one obtains a linear function of the form
\eq D_{h^1}\ldots D_{h^{d-1}} P_i(x)=\sum_{j=1}^n \Phi^i_j (h^1,\ldots,h^{d-1}) x_j,\ee
where $\Phi^i_j(h^1,\ldots,h^{d-1})$ is the multilinear form
\eq \Phi^i_j(h^1,\ldots,h^{d-1})= d!\,\sum_{1\leq j_1,\ldots,j_d\leq n} a^i_{j_1\ldots j_d,j}\ h^1_{j_1}\ldots h^{d-1}_{j_{d-1}},\ee
for any $(d-1)$-tuple of vectors $(h^1,\ldots,h^{d-1})\in\F_p^{(d-1)n}$. Note that on the diagonal
\eq \Phi^i_j(h,\ldots,h)= (d-1)!\ \partial_{x_j}P_i(h)\ee

\medskip

\noindent Define the set $W^*$ associated to the polynomial map $P$ by
\eq W^*=\{(h^1,\ldots,h^{d-1})\in\F_p^{(d-1)n}:\ rank\ (\Phi(h^1,\ldots,h^{d-1}))\,<\,R\},\ee
where $\Phi(h^1,\ldots,h^{d-1})$ is the $R\times n$ matrix with entries $\Phi^i_j(h^1,\ldots,h^{d-1})$ for $1\leq i\leq R$, $1\leq j\leq n$.

\medskip

\begin{proof}[\underline{Proof of Lemma 2}.] Using the definition of $U^d$ norm:

\eq
||e(\alpha\cdot P)||_{U^d}=\mathbb{E}(\Delta_{h_1,...,h_d}e(\alpha\cdot
P(x):\ x,h_1,...,h_d\in \mathbb{F}_p^n)\ee
\[\hspace{.81in} = \mathbb{E}(e(\alpha
\cdot D_{h_1,...,h_d}P(x)):x,h_1,...,h_d\in \mathbb{F}_p^n).\]

\medskip

\noindent From (2.2) (2.3) and the definition of the matrix $\Phi(h^1,\ldots,h^{d-1})$ it is clear that
\eq \Delta_{h_1,...,h_d}e(\alpha\cdot P(x))=e(\Phi^T(h^1,\ldots,h^{d-1})\al\cdot h^d)\ee
where $\Phi^T$ is the transpose of the matrix $\Phi$ and "$\cdot$" is the dot product. If $rank\,(\Phi(h^1,\ldots,h^{d-1}))=R$ then $\Phi^T(h^1,\ldots,h^{d-1})\,\al\neq 0$ hence summing (2.8) in the $h^d$ variable vanishes. Thus only the tuples $(h^1,\ldots,h^{d-1})\in W^*$ contribute to $||e(\alpha\cdot P)||_{U^d}$ and Lemma 2 follows.
\end{proof}

\medskip

If $\ rank\ (\Phi(h^1,\ldots,h^{d-1}))\,<\,R$ then its rows $\Phi^1,\ldots,\Phi^R$ are linearly dependent, thus one may write
\[W^*=\bigcup_{(\la_1,\ldots,\la_R)\neq 0} W^*(\la_1,\ldots,\la_R),\] where for $\la=(\la_1,\ldots,\la_R)\neq 0$
\eq W^*(\la)=\{(h^1,\ldots,h^{d-1})\in\F_p^{(d-1)n}:\ \la_1\Phi^1(h^1,\ldots,h^{d-1})+\ldots +\la_R \Phi^R(h^1,\ldots,h^{d-1})=0\}.\ee

Note that $W^*(\la)$ is a homogeneous algebraic set defined by $n$ equations of degree $d-1$. To estimate the size these sets we need the following basic facts from algebraic geometry.

\begin{lemma} \cite{GL} For a homogeneous (affine) algebraic set $U\subs\F_p^m$ of degree $r$ and dimension $s$ one has that
\eq |U|\leq r\,p^s\ee
\end{lemma}

The degree of the set $U$ is defined as the degree of its image $U^0$ in the $m-1$-dimensional projective space above $\F_p$. For projective algebraic sets it is shown in \cite{GL}, Prop. 12.1, that $|U^0|\leq r\,\pi_s(\F_p)$ where $\pi_s(\F_p)=(p^s-1)/(p-1)$ is the size of the $s-1$-dimensional projective space. This implies (2.10) as $|U|= (p-1)|U^0|+1$.

\begin{lemma} If a homogeneous algebraic set $U$ is defined by $n$ equations of degrees $d_1,\ldots,d_n$ then its degree is bounded by
\eq deg\,(U)\leq d_1d_2\ldots d_n\ee
\end{lemma}

The degree of a (projective) algebraic set is defined as the sum of degrees of its irreducible components, and for a projective algebraic variety it can be defined geometrically as the number of intersection points with a generic subspace of complementary dimension or algebraically, see \cite{harts} Prop. 7.6 and the definition preceding it. Note that the degree of a hypersurface is the degree of its defining polynomial.
Lemma 5 may be viewed as a generalization of Bezout's theorem on the number of intersection of plane algebraic curves, written as an inequality ignoring the multiplicities of the intersection points. It follows easily from the following basic inequality
which is an immediate corollary of Thm. 7.7 in \cite{harts}.

\begin{thmB} Let $Y$ be a (projective) variety of dimension at least 1 and let $H$ be a hypersurface not containing $Y$. Let $Z_1,\ldots,Z_s$ be the irreducible components of the intersection $Y\cap H$. Then
\eq \sum_{i=1}^s deg\,(Z_i)\leq deg\,(Y)\,deg\,(H)\ee
\end{thmB}

\begin{proof}[\underline{Proof of Lemma 5}.]
One may write (2.12) as
\eq deg\,(Y\cap H)\leq deg\,(Y)\,deg\,(H)\ee
as long as $Y\nsubseteq H$ and $dim\,Y\geq 1$. However for $Y\subs H$ or when $dim\,Y=0$ ($Y$ being a point and $deg\,(Y)=1$) inequality (2.13) holds trivially. It also extends to algebraic sets $V$ by writing them as union of their irreducible components $Y$ and using (2.13) for each $Y$ together with the fact that each irreducible component of $V\cap H$ is contained in some $Y\cap H$ for an irreducible $Y\subs V$. Then (2.11) follows immediately by induction on $n$.
\end{proof}

\begin{proof}[\underline{Proof of Lemma 3}.]
For a given $\la\in\F_p^R\backslash\{0\}$ one has that $dim\,W^*(\la)\leq dim\,W^*\leq (d-1)n-K$, and since it is defined by $n$ equations of degree $d-1$ its degree satisfies the bound $deg\,(W^*(\la)\leq (d-1)^n$. Hence by (2.10) and (2.11) one has for all $\la\in\F_p^R\backslash\{0\}$
\[|W^*(\la)|\leq (d-1)^n p^{(d-1)n-K}\] and Lemma 3 follows from the decomposition $W^*=\cup_\la W^*(\la)$.
\end{proof}

\section{Generic polynomial maps.}

We will work above the algebraically closed field of characteristic $p$, denoted by $\bar{\F}_p$. Though it is best to view the arguments below by identifying homogeneous algebraic sets with their image in the projective space, we will keep to the affine settings.
Let $\mathcal{P}(n,d)$ be the space of homogeneous polynomials $P:\bF_p^n\to\bF_p$ of degree $d$, which is the $N={n+d-1 \choose d}$-dimensional linear space over $\bF_p$. The critical locus $\mathcal{C}\subs \mathcal{P}(n,d)\times \bF_p^n$ as the set of pairs $(P,x)$ such that $x\in S_P^*$, the singular variety of $S_P=P^{-1}(0)$, that is $\partial_{x_1}P(x)=\ldots =\partial_{x_n}P(x)=0$. Its image under the natural projection $\pi:\mathcal{P}(n,d)\times \bF_p^n\to \mathcal{P}(n,d)$ is the discriminant $\mathcal{D}=\mathcal{D} (n,d)$, which is the locus of polynomials $P$ so that $S_P$ is singular. It is known that $\mathcal{C}$ is a connected smooth variety of dimension equal to that of $D$, which is a hypersurface in $\mathcal{P}(n,d)$, see \cite{Varl}, Cor.2 and Cor.5.

For given $K\geq 1$ let $\mathcal{D}_K=\mathcal{D}_K(n,d)$ be the locus of polynomials $P$ such that $dim\,S_P^*\geq K$, and let $\mathcal{C}_K=\pi^{-1} (D_K)$ be its pre-image. The set $D_K$ is a Zarisky closed set, moreover the pre-image $\pi^{-1}(P)=S_P^*$ of every point $P\in\mathcal{D}_K$ has dimension at least $K$ (at least dimension $K-1$ in the projective settings). This implies that
\eq dim\ \mathcal{D}_K \leq dim\ \mathcal{C}-K+1=N-K.\ee
Thus $\mathcal{D}_K$ has codimension at least $K$.

\begin{proof}[\underline{Proof of Proposition \ref{generic}}.]

For given $R$ let $\mathcal{P}(n,d)^R$ the space of homogeneous polynomial maps $P=(P_1,\ldots,P_R):\bF_p^n\to\bF_p^R$ of degree $d$. If $x\in S_P^*$ then by definition (1.2), these exists a nonzero $\mu=(\mu_1,\ldots,\mu_R)\in\bF_p^R$, such that $x\in S_{P_\mu}^*$, where $P_\mu=\mu_1 P_1+\ldots +\mu_R P_R$. Note that $S_{P_\mu}=S_{P_{\mu'}}$ if $\mu=\la\mu'$ for a scalar $\la\neq 0$. Thus

\eq S_P^*\subs \bigcup_{\mu\in \Pi_p^{R-1}} S_{P_\mu}^*,\ee
where  $\Pi_p^{R-1}$ is the $R-1$ dimensional projective space above $\bar{\F}_p$, considered as an algebraic set in $\bF_p^R$. This implies that for a given $L\geq 2R$, if $dim\,S_P^*\geq L$ then there must exist a $\mu\neq 0$ such that $dim\,S_{P_\mu}\geq L-R+1$.

Let $\Phi:\Pi_p^{R-1}\times \mathcal{P}(n,d)^R\to \mathcal{P}(n,d)$ be the map defined by $\Phi (\mu,P)=P_\mu$, and let $\pi: \Pi_p^{R-1}\times \mathcal{P}(n,d)^R\to \mathcal{P}(n,d)^R$ be the natural projection. Then the locus of polynomial maps $P$ for which $dim\,S_P^*\geq L$ is contained in
\eq \{P\in\mathcal{P}(n,d)^R:\ dim\,S_P^*\geq L\}\subs \pi(\Phi^{-1}D_K)\ee
with $K=L-R+1$.
The tangent map $d\Phi_{(\mu,P)}$ is onto at every point $(\mu,P)$ where $\mu\neq 0$, thus the codimension of the algebraic set $\Phi^{-1}D_K$ is at least $K$. The projection $\pi$ cannot increase the dimension, hence the codimension of $\pi(\Phi^{-1}D_K\subs \mathcal{P}(n,d)^R$ is at least $K-R+1$. If $L\geq 2R-1$ then the set of polynomial maps $P$ for which $dim\ S_P^*\geq L$ has codimension at least $L-2R+2\geq 1$, thus is contained in a proper algebraic set. This proves Proposition \ref{generic}.
\end{proof}

\section{Linear subspaces in homogeneous varieties.}

We will show that a homogeneous variety $S=P^{-1}(0)$ contains a large linear subspace $M$, as an easy corollary of the following result due to Chevalley and Warning (\cite{FF}, Thm. 6.11)

\begin{thmC} Let $Q_1,\ldots,Q_t:\F_p^n\to\F_p$ be polynomials of degree $d_1,\ldots,d_t$ such that $D=d_1+\ldots +d_t<n$, and $Q_i(0)=0$ for all $i$. If $S_Q$ is the common zero set of the polynomials $Q_i$ then
\eq |S_Q|\geq p^{n-D}\ee
\end{thmC}

\begin{prop} Let $S=P^{-1}(0)$, where $P:\F_p^n\to F_p^R$ a homogeneous polynomial map of degree $d<p$. Then
$S$ contains a linear subspace $M$ such that
\eq dim\,M\geq c_d (n/R)^{\frac{1}{d}}\ee
with a constant $c_d>0$ depending only on $d$.
\end{prop}

\begin{proof} Let $M$ be a maximal subspace, such that $M\subs S$. Let $h_1,\ldots,h_m$ be a basis of $M$. One may write $P_i(x)=Q_i(x,\ldots,x)$ where $Q_i(x_1,\ldots,x_d)$ is a symmetric multi-linear form, as in (2.1).
If $h$ is such that $Q_i(h,\ldots,h,h_{i_{k+1}},\ldots,h_{i_d})=0$ for all $1\leq k\leq d$, $1\leq i\leq R$, and $1\leq i_{k+1}\leq\ldots,\leq i_d\leq m$, then $M'=M+\F_p h\subs S$ as well. For fixed $k$ this gives $R\,{m \choose k}$ homogeneous equations of degree $k$. The sum of degrees $D$ of all these equations is bounded by
\eq D\leq C_d R\,m^d\ee
By the Chevalley-Warning's Theorem, the number of such $h$ is at least $p^{n-D}$. If $m<c_d\,(m/R)^{\frac{1}{d}}$, then $p^{n-D}>p^m=|M|$. Thus one may choose $h\notin M$ such that $M+\F_p h\subs S$ contradicting our assumption. This proves the Proposition.
\end{proof}

\begin{cor} Let $A\subs\F_p^n$ of density $\de>0$, and let $S=P^{-1}(0)$, where $P:\F_p^n\to F_p^R$ a homogeneous polynomial map of degree $d<p$. If
\eq R<c(\de,l,d,p)\,n\ee
then $A$ contains an arithmetic progression $\{x,x+y,\ldots,x+(l-1)y$\} with common difference $y\in S\backslash\{0\}$.
\end{cor}

\begin{proof} Let $M$ be a maximal subspace contained in $S$. Let $M+x_i$ be a translate of $M$ such that the relative density $\de_i=|(A\cap (M+x_i)|/|M|$ of $A$ on $m+x_i$ is at least $\de$. If the dimension $m$ of $M$ is large enough: $m\geq m(\de,l,p)$ then $A\cap (M+x_i)$ contains a non-trivial arithmetic progression of length $l$, whose gap $y$ is then in $M\subs S$. By (4.2) this happens if $(n/R)^{1/d}>m(\de,l,d,p)$ for which it is enough to assume (4.4).
\end{proof}



\bigskip
\bigskip

\begin{thebibliography}{12}

\bigskip

\bibitem{Birch}
{\sc J. Birch} {\em Forms in many variables}, Proc. Roy. Soc. Ser. A 265 (1962), 245-263

\medskip

\bibitem{BO}
{\sc J. Bourgain} {\em A Szemer\'edi type theorem for sets of positive density in  $\R^k$}, Israeli J. Math, 54 (1986), 307-316

\medskip

\bibitem{G}
{\sc B. Green} {\em On arithmetic structures in dense sets of integers}, Duke Math. J. 114 (2) (2002), 215-238

\medskip

\bibitem{FF}
{\sc R. Lidl and H. Niederreiter} Finite Fields, Cambridge University Press (1997)

\medskip

\bibitem{GT}
{\sc B. Green and T. Tao}, {\em Linear equations in the primes}, Annals of Math. (to appear)

\medskip

\bibitem{GL}
{\sc S. Ghorpage and G. Lachaud} {\em \'Etale cohomology, Lefshetz theorems and number of points of singular verieties over finite fields}, Moscow Math. J. (2) (2002), 589-631

\medskip

\bibitem{harts}
{\sc R. Hartshorne}, {\em Algebraic Geometry}, Graduate texts in mathematics: 52, Springer-Verlag (1977)

\medskip

\bibitem{S}
{\sc A. S\'ark\"ozy}, {\em On difference sets of sequences of integers III}, Acte Math. Acad. Sci. Hungar. 31 (1978), 355-386

\medskip

\bibitem{SZ}
{\sc E. Szemer\'{e}di}, {\em On sets of integers containing no k elements in arithmetic progression}, Acta Arith. 27 (1975), 299-345

\medskip

\bibitem{taovu}
{\sc T. Tao and V. Vu} {\em Additive combinatorics}, Cambridge University Press (2004)

\medskip

\bibitem{Varl}
{\sc R. Smith and R. Varley} {\em The tangent cone to the discriminant}, Proc.  Conf. in Alg. Geom. Vancouver (1984)

\medskip

\bibitem{WZ}
{\sc T. Wooley and T. Ziegler}, {\em Multiple recurrence and convergence along the primes}, preprint (2010)

\end{thebibliography}
\end{document}